\documentclass[12pt]{amsart}

\usepackage{amssymb,amsbsy,amsmath,amsfonts,amssymb,amscd}
\usepackage{latexsym}

\input xy
\xyoption{all}

\newcommand\sL{{\mathcal L}}

\DeclareMathOperator{\id}{id}

\newcommand{\CC}{\ensuremath{\mathbb{C}}}

\newcommand{\ZZ}{\ensuremath{\mathbb{Z}}}

\newcommand{\HH}{\ensuremath{\mathbb{H}}}
\newcommand{\PP}{\ensuremath{\mathbb{P}}}

\newcommand{\FF}{\ensuremath{\mathbb{F}}}

\newcommand{\ra}{\ensuremath{\rightarrow}}

\def\eea{\end{eqnarray*}}
\def\bea{\begin{eqnarray*}}

\newcommand\dual{\mathrel{\raise3pt\hbox{$\underline{\mathrm{\thinspace d
\thinspace}}$}}}

\newcommand\QED{\ifhmode\unskip\nobreak\fi\quad {\rm Q.E.D.}} 
\newcommand\qe{\ifhmode\unskip\nobreak\fi\quad $\Box$}       

\def\BOX{\hfill\lower.5\baselineskip\hbox{$\Box$}}

\newtheorem{theo}{Theorem}[section]
\newtheorem{remarkk}[theo]{Remark}
\newenvironment{rem}{\begin{remarkk}\rm}{\end{remarkk}}

\newtheorem{defin}[theo]{Definition}
\newenvironment{definition}{\begin{defin}\rm}{\end{defin}}

\newtheorem{prop}[theo] {Proposition}

\newtheorem{lemma}[theo]{Lemma}
\newtheorem{example}[theo]{Example}

\DeclareMathOperator{\Tors}{Tors}

\begin{document}

\title[New Burniat-type  surfaces ]{Burniat-type surfaces and a new family of surfaces with $p_g = 0$, $K^2=3$}
\author{I. Bauer, F. Catanese}

\thanks{The present work took place in the realm of the DFG
Forschergruppe 790 "Classification of algebraic surfaces and
compact complex manifolds". A large part of this work was done while the authors were guests at KIAS, Seoul in April 2012: we are grateful
to KIAS for the hospitality and wonderful working environment.}

\date{\today}

\maketitle

\section*{Introduction}

The present paper continues a research developed in a series of articles (  \cite{bacat},   \cite{bcg},\cite{4names},\cite{keumnaie}, \cite{burniat1},  
\cite{burniat2}, \cite{burniat3},  \cite{bacainoue}) dedicated to the  classification,  the moduli spaces and the discovery of new surfaces 
of general type with geometric genus
$p_g = 0$ (the first such having been constructed in \cite{Cam} and \cite{god}), with particular emphasis on the problem of classifying the possible fundamental groups occurring
according to the respective values of $K^2_{S_{min}}$ (see \cite{milesLNM}, \cite{survey}, \cite{mlp1}, \cite{mlp},
 \cite{mlpr}, \cite{4names} for related conjectures and results). 

The construction methods we have been using vary considerably, and in this paper we consider the method originally due to Burniat
(Abelian coverings)
in  the reformulation done by Inoue (quotients by Abelian groups), presenting it in a  very general fashion which seems worthwhile a deeper investigation.

Our general approach  consists in  considering quotients (cf. \cite{bacainoue} for  the case of a free action), by some group $G$ of the form $(\ZZ/m)^r$, of varieties $\hat{X}$ contained in a product of curves $\Pi_i C_i$, where each $C_i$  is  a maximal Abelian cover
of the projective line with Galois  group of the form $(\ZZ/m)^{n_i}$. Let us explain now the connection with Burniat surfaces.

Burniat surfaces are surfaces of general type with invariants $p_g=0$ and $K^2 = 6,5,4,3,2$, whose birational
models  were constructed by Pol Burniat (cf. \cite{burniat}) in 1966 as singular bidouble covers of the projective plane. Later these surfaces were reconstructed by Inoue (cf. \cite{inoue}) as $G:=(\ZZ/2\ZZ)^3$-quotients of a ($G$-invariant) hypersurface $\hat{X}$ of multi degree $(2,2,2)$ in a product of three elliptic curves. In the case where $G$ acts freely, this construction and its  topological characterization
 has been largely generalized by the authors in the already cited paper  \cite{bacainoue}. 

While Inoue writes the (affine) equation of $\hat{X}$ in terms of the uniformizing parameters of the respective elliptic curves using a variant of the Weierstrass' functions (the Legendre functions), we found it much more useful, especially for a systematic approach to finding all possible such constructions, to write the elliptic curves as a complete intersection of two diagonal quadrics in three space.
In fact, we consider the following diagram:

\begin{equation}\label{burniatdiagr}
\xymatrix{
E_1 \times E_2 \times E_3 \ar[dd]_{\mathcal{H}':=(\ZZ/2)^3}^{\pi'}& E_1: x_1^2 + x_2^2 + x_3^2 = 0, \  x_0^2 = a_1x_1^2 + a_2x_2^2+a_3x_3^2\\
&E_2: u_1^2 + u_2^2 + u_3^2 = 0, \ u_0^2 = b_1u_1^2 + b_2u_2^2+b_3u_3^2\\
P_1:=\PP^1 \times \PP^1 \times \PP^1\ar[dd]^{\pi }_{\mathcal{H}:=((\ZZ/2)^2)^3} & E_3: z_1^2 +z_2^2 +z_3^2 = 0,  \ z_0^2 = c_1 z_1^2 + c_2z_2^2 + c_3 z_3^2\\
&\\
P_2:=\PP^1 \times \PP^1 \times \PP^1
}
\end{equation}
We consider then $P_1$ with homogeneous coordinates $((s_1:t_1),(s_2:t_2),(s_3:t_3))$ and the pencil of Del Pezzo surfaces
of degree $6$ 
$$
Y_{\lambda}:= \{s_1s_2s_3 = \lambda t_1t_2t_3\} \subset P_1.
$$
$Y_{\lambda}$ is invariant under a subgroup $H_0 \cong (\ZZ / 2 \ZZ)^2$ of $\mathcal{H}$  generated by the transformations:
$$
t_i \mapsto \epsilon_i t_i, \ \ , \epsilon \in \{ \pm 1\},  \ \ \epsilon_1 \epsilon_2 \epsilon_3 = 1.
$$
Therefore $\hat{X}_{\lambda}:= (\pi')^{-1}(Y_{\lambda})$ is invariant under a subgroup $G_1 \cong (\ZZ /2 \ZZ)^5 \subset (\ZZ /2 \ZZ)^9$. It is now our aim to find 
all subgroups $G \cong (\ZZ/2\ZZ)^3 \subset G_1$, with the property  that $G$ acts freely on $\hat{X}$.

We give the following 
\begin{definition}
Let $G \cong (\ZZ / 2 \ZZ)^3 \leq G_1$, such that $G$ acts freely on $\hat{X}_{\lambda}$. Then $S_{\lambda}:=\hat{X}_{\lambda} /G$   is called a {\em primary Burniat type surface}. 
\end{definition}
Obviously, primary Burniat surfaces (i.e., Burniat surfaces with $K^2=6$) are primary Burniat type surfaces.
With the help of the computer algebra system MAGMA  we can classify all primary Burniat type surfaces and can 
prove the following

\begin{theo}
Primary Burniat type surfaces are exactly the primary Burniat surfaces.
\end{theo}

We then consider $\hat{X}:=(\pi')^{-1}(Y_1)$. 

Since $Y_1$ is invariant under a bigger subgroup of $\mathcal{H}$ it turns out that $\hat{X}$ is invariant under $G_0 \cong (\ZZ/2 \ZZ)^6$. 

In the second part of the paper we find all subgroups $G \cong (\ZZ/ 2\ZZ)^4 \leq G_0$ with the property that there is exactly one non trivial element $g \in G$ such that $g$ has isolated fixed points on $\hat{X}$ and all other non trivial elements of $G$ act freely on $\hat{X}$. The quotient of $\hat{X}$ under $G$ is then a surface having exactly four ordinary nodes and we give the following

\begin{definition}
 Let $G \cong (\ZZ / 2 \ZZ)^4 \leq G_0$, be such that  $\hat{X}_1 = \hat{X}$ is $G$-invariant. We shall say that $G$ acts {\em 1-almost freely} on $\hat{X}$, if there is exactly one nontrivial element $g_0 \in G$ having isolated fixed points on $\hat{X}$, and all the other  act freely.

 Then the minimal resolution  $S$ of  the nodal surface $X:=\hat{X} /G$ is  called a {\em 4-nodal Burniat type surface}. 
\end{definition}

We give a complete classification of 4-nodal Burniat type surfaces, which turn out to be minimal surfaces of general type with $p_g=0$ and $K^2 =3$.
This gives us a list of seven subgroups $G$ yielding three (3-dimensional) families of such surfaces. Since these families are nowhere dense in the  moduli space,
and also in order to determine whether this constructions yields  hitherto unknown surfaces, we use a result of Armstrong to calculate the fundamental groups of these surfaces.

We see that  these families yield  three different topological types:  one family yields the same fundamental group as   the family of 
Keum-Naie surfaces with $K^2 =3$, that is case (i), one yields (case ii))  tertiary Burniat surfaces with $K^2 =3$, and the third family  
 realizes a new fundamental group $P:=SmallGroup(16,13)$ (case iii)). 
 Observe that $P$ is the central product of the dihedral group of order 8 and the cyclic group of order 4.

We summarize our result as follows:
\begin{theo}
Let $S$ be a 4-nodal Burniat type surface. Then $S$ is a minimal surface of general type with $K_{S}^2 = 3$, $p_g(S) = 0$, with one of the following fundamental groups of order $16$:
\begin{itemize}
\item[i)] $\pi_1(S) \cong(\ZZ / 2 \ZZ)^2 \times \ZZ/ 4 \ZZ$, or
\item[ii)] $\pi_1(S) \cong \HH_8 \times \ZZ / 2 \ZZ$, or
\item[iii)] $\pi_1(S) \cong SmallGroup(16,13)$.
\end{itemize}
\end{theo}

In a sequel to this paper we shall give other applications of the method considered here, constructing new surfaces as quotients of 
subvarieties of products of maximal abelian coverings of $\PP^1$ having Galois  group of the form $(\ZZ/d)^{m_i}$.

\section{Burniat surfaces  as reconstructed by Inoue}

We briefly recall the construction of the  {\em Burniat surfaces} (cf. \cite{burniat}) as given by Inoue (\cite{inoue}).
The description given by Inoue is very appropriate in order to calculate the fundamental group.

For 
$i \in \{1,2,3\}$, let $E_i := \CC / \langle 1, \tau_i \rangle$ 
be 
a complex elliptic curve. Denoting  by $z_i$ a
uniformizing 
parameter on $E_i$, we consider the following three 
involutions on 
$T:=E_1 \times E_2\times E_3$:
\begin{itemize}
\item[] $ g_1(z_1,z_2,z_3) = 
(-z_1+\frac 12, z_2 + \frac 12, z_3)$,
\item[] $g_2(z_1,z_2,z_3) 
= (z_1, -z_2 + \frac 12, z_3 +\frac 12)$,
\item[] $g_3(z_1,z_2,z_3) = (z_1+\frac 12, z_2, -z_3 + \frac 12)$.

\end{itemize}
 Then $G:= \langle g_1,g_2,g_3 \rangle \cong (\ZZ / 2 \ZZ)^3$.

 We consider the Legendre $\sL$-function for $E_i$ and denote it by 
$\sL_i$,  for $i=1,2,3$:
$\sL_i$ is a meromorphic
function on $E_i$ and $\sL_i \colon E_i \ra \PP^1$ is a double cover 
ramified in
$\pm 1, \pm a_i \in \PP^1 \setminus \{0, \infty\}$.
It is well known that we have (cf. \cite{inoue}, lemma 3-2, and also 
cf. \cite{burniat1}, pages 52-54, section 1 for an algebraic 
treatment):
\begin{itemize}
   \item[-] $\sL_i(\frac 12) = -1$, $\sL_i(0) = 1$,
$\sL_i(\frac{\tau_i}{2}) = a_i$,  $\sL_i(\frac{1+\tau_i}{2}) = - 
a_i$;
\item[-] let $b_i:= \sL_i(\frac{\tau_i}{4})$: then $b_i^2 = 
a_i$;
\item[-] $\frac{\rm{d} \sL_i}{\rm{d} z_i} (z_i) = 0$ if and 
only if 
$z_i \in \{0, \frac 12,\frac{\tau_i}{2}, \frac{1+\tau_i}{2} 
\}$ since these are the ramification points of $\sL_i$.
\end{itemize}
Moreover 
$$
\sL_i(z_i) = \sL_i(z_i+1) = \sL_i(z_i+\tau_i) = \sL_i(-z_i) = 
-\sL_i(z_i+ \frac 12),
$$
$$
\sL_i(z_i + \frac{\tau_i}{2}) = 
\frac{a_i}{\sL_i(z_i)}.
$$
Consider
$$
\hat{X}_c := \{ (z_1,z_2,z_3)) \in T \ | \ \sL_1(z_1)
\sL_2(z_2)\sL_3(z_3)  = c,\}.
$$ Then
\begin{itemize}
\item[-] $\hat{X}_c$ is invariant under the action of $G$, 
   \item[-] for a general choice of $c$, $\hat{X}_c$ is a smooth hypersurface in $T$ of multidegree $(2,2,2)$  and $G$ acts
freely on $\hat{X}_c$, thus $X_c:=\hat{X}_c/G$ is a smooth minimal surface of general type with $p_g=0$, $K^2 =6$.
\item[-] for special values of $c$ the hypersurface $X_c$ has $4,8,12,16$ nodes, which are isolated fixed points of $G$; in these cases the minimal resolution of singularities of $X_c := \hat{X}_c/G$ is a minimal surface of general type with $p_g=0$ and $K^2 =5,4,3,2$.

\end{itemize}

\section{Intersection of diagonal quadrics and $(\ZZ/ 2 \ZZ)^n$-actions}\label{actions}

We consider  diagram (\ref{burniatdiagr}):

\begin{equation*}
\xymatrix{
E_1 \times E_2 \times E_3 \ar[dd]_{\mathcal{H}':=(\ZZ/2)^3}^{\pi'}& E_1: x_1^2 + x_2^2 + x_3^2 = 0, \  x_0^2 = a_1x_1^2 + a_2x_2^2+a_3x_3^2\\
&E_2: u_1^2 + u_2^2 + u_3^2 = 0, \ u_0^2 = b_1u_1^2 + b_2u_2^2+b_3u_3^2\\
P_1:=\PP^1 \times \PP^1 \times \PP^1\ar[dd]^{\pi }_{\mathcal{H}:=((\ZZ/2)^2)^3} & E_3: z_1^2 +z_2^2 +z_3^2 = 0,  \ z_0^2 = c_1 z_1^2 + c_2z_2^2 + c_3 z_3^2\\
&\\
P_2:=\PP^1 \times \PP^1 \times \PP^1
}
\end{equation*}

\begin{rem}
1) $\pi'$ is given by ` forgetting' the variables $x_0, u_0, z_0$.

2) $\pi$ is given by $x_i^2 = y_i$, $u_i^2 = v_i$, $z_i^2 = w_i$, $i=1,2,3$, where we consider 
$$
P_2 \subset \PP^2 \times \PP^2 \times \PP^2
$$
as given by the equations 
$$
y_1 +y_2+y_3=0, \  \  v_1 +v_2+v_3=0, \  \  w_1 +w_2+w_3=0. 
$$
3) The inverse image of the Del Pezzo surface $Y'_{\lambda}:=\{y_1v_1w_1 = \lambda y_2v_2w_2\} \subset P_2$ under $\pi$  splits in two irreducible components:

$$
\pi^{-1}(\{y_1v_1w_1 = \lambda y_2v_2w_2\}) = Y_{\lambda}^+ \cup Y_{\lambda}^- \subset P_1,
$$
where $Y_{\lambda}^{\pm}:= \{x_1u_1z_1 = \pm \sqrt{\lambda} x_2u_2z_2\}$.

4) If we take   homogeneous coordinates 
$$
((s_1:t_1),(s_2:t_2),(s_3:t_3)),
$$
such that  the action of $\mathcal{H}$ on $P_1=(\PP^1)^3$ 
 is generated by the transformations:
$$
t_i \mapsto \pm t_i, \ \ , s_i \mapsto s_i, \ \ \ 1 \leq i \leq3,
$$
$$
t_i \mapsto s_i, \ \ s_i \mapsto t_i, \ \ \  1 \leq i \leq3,
$$
then we see that the Del Pezzo surface 
$$
Y_{\lambda}:= \{s_1s_2s_3 = \lambda t_1t_2t_3\} \subset P_1
$$
is invariant under the subgroup $H_0 \cong (\ZZ / 2 \ZZ)^2$ of $\mathcal{H}$ generated by the transformations:
$$
s_i \mapsto s_i, \ \ t_i \mapsto \epsilon_i t_i, \ \ , \epsilon \in \{ \pm 1\},  \ \ \epsilon_1 \epsilon_2 \epsilon_3 = 1.
$$

Then $\hat{X}_{\lambda}:= \pi'^{-1}(Y_{\lambda})$ is invariant under $G_1 \cong (\ZZ /2 \ZZ)^5 \subset (\ZZ /2 \ZZ)^9$. It is now our aim to find 
all subgroups $G \cong (\ZZ/2\ZZ)^3 \subset G_1$, such that $G$ acts freely on $\hat{X}_{\lambda}$.

We obtain in this case a commutative diagram

\begin{equation}\label{surface}
\xymatrix{
\hat{X}_{\lambda}  \ar[d]_{(\ZZ/2)^3}^{\pi'} \ar[dr]^{(\ZZ/2)^3 \cong G}\\
Y_{\lambda} \subset P_1 \ar[d]_{H_0} & S_{\lambda}:=\hat{X}_{\lambda}/G \ar[dl]^{(\ZZ/2)^2}\\
Z' .
}
\end{equation}
$S$ is then a smooth minimal surface of general type with $K_S^2 =6$, $p_g = 0$. We shall in fact show that necessarily $S$ is a primary Burniat surface.

5) If instead we set $\lambda =1$, we see that the
Del Pezzo surface 
$$
Y:= Y_1=\{s_1s_2s_3 =  t_1t_2t_3\} \subset P_1
$$
is invariant under the subgroup $H_1 \cong (\ZZ / 2 \ZZ)^3$ of $\mathcal{H}$ generated by the transformations:
$$
s_i \mapsto s_i, \ \  t_i \mapsto \epsilon_i t_i, \ \ , \epsilon \in \{ \pm 1\},  \ \ \epsilon_1 \epsilon_2 \epsilon_3 = 1
$$
and
$$
t_i \mapsto s_i, \ \ s_i \mapsto t_i, \ \ \  \forall i .
$$

Then $\hat{X}:= \pi'^{-1}(Y)$ is invariant under $G_0 \cong (\ZZ /2 \ZZ)^6 \subset (\ZZ /2 \ZZ)^9$. It is now our aim to find 
all subgroups $G \cong (\ZZ/2\ZZ)^4 \subset G_0$, such that there is exactly one nontrivial 
element $g_0 \in G$, which has isolated fixed points on $\hat{X}$ and all other nontrivial elements act freely.

We obtain then a commutative diagram

\begin{equation}\label{surface2}
\xymatrix{
\hat{X}  \ar[d]_{(\ZZ/2)^3}^{\pi'} \ar[dr]^{(\ZZ/2)^4 \cong G}\\
Y \subset P_1 \ar[d]_{H_1} & S:=\hat{X}/G \ar[dl]^{(\ZZ/2)^2}\\
Z .
}
\end{equation}
6) Note that it is easy to see that $Z$ is the four nodal cubic surface in $\PP^3$. In fact, $Y^{\pm} =\{s_1s_2s_3 =  \pm  t_1t_2t_3\}$ is the pull-back of $Z':= \{\sigma_1 \sigma_2 \sigma_3 = \tau_1 \tau_2 \tau_3 \}$ under the map $s_i^2 = \sigma_i$, $t_i^2 = \tau_i$. Hence $Y^+ \rightarrow Z'$ is a $(\ZZ / 2 \ZZ)^2$-cover of a Del Pezzo surface of degree $6$. On $Z'$, the involution which exchanges $\sigma_i$ and $\tau_i$ has four isolated fixed points. Hence the quotient $Z$ is a four nodal cubic surface.
\end{rem}

Observe that in the above remark we described the action of $\mathcal{H}$ on $(\PP^1)^3$ in the coordinates $(s_i:t_i)$. 
We have to rewrite this action  in the coordinates $x_i, u_i, z_i$, and then give the equations of the Del Pezzo surface
$$
Y_{\lambda}=  \{s_1s_2s_3 = \lambda t_1t_2t_3\} \subset P_1
$$
in the coordinates $x_i, u_i, z_i$.

In order to give the action of $\mathcal{H}$ in the coordinates $x_i, u_i, z_i$ and find the equations of the Del Pezzo surfaces $Y_{\lambda} \subset P_1 \subset (\PP^2)^3$, we consider first the following diagram
\begin{equation}\label{onecurve}
\xymatrix{
E_1 = E \ar[d]_{\ZZ/2 \ZZ}\\
\PP^1 \ar[d]_{(\ZZ/2 \ZZ)^2} &= & \{x_1^2 +x_2^2 +x_3^2 = 0\} =:C \subset \PP^2\\
\PP^1 & = &\{y_1 + y_2 + y_3 = 0 \} \subset \PP^2.
}
\end{equation}
Observe that 
$$
x_1^2 + x_2^2 +x_3^2 = 0 \iff \det \begin{pmatrix}
x_1+ix_2 & -x_3\\x_3&x_1 -i x_2
\end{pmatrix} = 0.
$$

Therefore we get a parametrization of $C$: 
$$(s:t) = (x_1 + ix_2:x_3) = (-x_3:x_1-ix_2).
$$

With this parametrization, we can rewrite the action of $(\ZZ / 2 \ZZ)^2$ on $\PP^1$ (we use the convenient  notation
by which all variables not mentioned in a transformation are left unchanged):

\begin{itemize}
\item[a)] $x_1 \mapsto -x_1$ (or equivalently $\begin{pmatrix} x_2 \\ x_3 \end{pmatrix} \mapsto \begin{pmatrix} -x_2 \\ -x_3 \end{pmatrix}$) corresponds to $(s:t) \mapsto (t:s)$;
\item[b)] $x_2 \mapsto -x_2$ (or equivalently $\begin{pmatrix} x_1 \\ x_3 \end{pmatrix} \mapsto \begin{pmatrix} -x_1 \\ -x_3 \end{pmatrix}$) corresponds to $(s:t) \mapsto (-t:s)$;
\item[c)] $x_3 \mapsto -x_3$ (or equivalently $\begin{pmatrix} x_1 \\ x_2 \end{pmatrix} \mapsto \begin{pmatrix} -x_1 \\ -x_2 \end{pmatrix}$) corresponds to $(s:t) \mapsto (s:-t)$.

\end{itemize}

\begin{rem}\label{paramfp}
The fixed points of the three involutions above are
\begin{itemize}
\item[a)] $s= \pm t \ \iff x_1 = x_3 \pm i x_2 = 0$;
\item[b)] $t = \pm is \ \iff x_2 = x_1 \pm ix_3 = 0$;
\item[c)] $st=0 \ \iff x_3 = x_1 \pm ix_2 = 0$.
\end{itemize}
\end{rem}
The equations for the Del Pezzo surface $Y_{\lambda}=  \{s_1s_2s_3 = \lambda t_1t_2t_3\}$ in the coordinates $x_i$, $u_i$, $z_i$ can now be easily computed.

\begin{lemma}\label{8equations}  
Consider 
$$\PP^1 \times \PP^1 \times \PP^1 \subset \PP^2_{(x_1:x_2:x_3)} \times \PP^2_{(u_1:u_2:u_3)}  \times \PP^2_{(z_1:z_2:z_3)} ,$$
given by the equations
$$
x_1^2 + x_2^2+x_3^2 =0, \ \  u_1^2 + u_2^2+u_3^2 =0, \ \  z_1^2 + z_2^2+z_3^2 =0.
$$

Let $Y_{\lambda} \subset \PP^1_{(s_1:t_1)} \times \PP^1_{(s_2:t_2)} \times \PP^1_{(s_3:t_3)}$ be the Del Pezzo surface given by the equation 
$$
Y=  \{s_1s_2s_3 = \lambda t_1t_2t_3\}, \ \lambda \neq 0.
$$

Then 
$$
Y_{\lambda} \subset \PP^1 \times \PP^1 \times \PP^1 \subset \PP^2_{(x_1:x_2:x_3)} \times \PP^2_{(u_1:u_2:u_3)}  \times \PP^2_{(z_1:z_2:z_3)} 
$$
is given by the following 8 equations:
\begin{enumerate}\label{equationsdp}
\item $(x_1+ix_2)(u_1+iu_2)(z_1+iz_2) = \lambda x_3u_3z_3$,
\item $(x_1+ix_2)(u_1+iu_2)(-z_3) = \lambda x_3u_3(z_1-iz_2)$,
\item $(x_1+ix_2)(-u_3)(z_1+iz_2) = \lambda x_3(u_1-iu_2)z_3$,
\item $(-x_3)(u_1+iu_2)(z_1+iz_2) = \lambda (x_1-iu_2)u_3z_3$,
\item $(x_1+ix_2)u_3z_3 = \lambda x_3(u_1-iu_2)(z_1-iz_2)$,
\item $x_3(u_1+iu_2)z_3 = \lambda (x_1-ix_2)u_3(z_1-iz_2)$,
\item $x_3u_3(z_1+iz_2) = \lambda (x_1-ix_2)(u_1-iu_2)z_3$,
\item $-x_3u_3z_3 = \lambda (x_1-ix_2)(u_1-iu_2)(z_1-iz_2)$.
\end{enumerate}

\end{lemma}

\begin{proof}
We have seen that each $\PP^1$ (written as a conic in $\PP^2$) has a birational map to $\PP^1$ given by:
$$
(s_1:t_1) = (x_1 + ix_2:x_3) = (-x_3:x_1-ix_2),   
$$
$$(s_2:t_2) = (u_1 + iu_2:u_3) = (-u_3:u_1-iu_2),
$$
$$(s_3:t_3) = (z_1 + iz_2:z_3) = (-z_3:z_1-iz_2).
$$
Since each birational map is  well defined at each point either through the second or through the third ratio,
this implies immediately that the divisorial equation of the Del Pezzo surface in $ \PP^1 \times \PP^1 \times \PP^1 $
is equivalent to  the above eight equations in $\PP^2 \times \PP^2  \times \PP^2$.

\end{proof}

Let $\hat{X} \subset E_1 \times E_2 \times E_3$ be the inverse image of the Del Pezzo surface $Y_{\lambda} \subset P_1$ given by the above eight equations. Then we have:

\begin{lemma}
1)  \underline{$\lambda \neq 0$:} then
$\hat{X}_{\lambda}$ is invariant under the group $G_1 \cong (\ZZ / 2 \ZZ)^5 \leq (\ZZ / 2 \ZZ)^3 \times (\ZZ / 2 \ZZ)^3 \times (\ZZ / 2 \ZZ)^3$, where

$$
G_1:= \{ (\epsilon_0,  \epsilon_1, \eta_0, \epsilon_2, \zeta_0, \epsilon_3) \subset (\ZZ / 2 \ZZ)^6 | \epsilon_1 \epsilon_2 \epsilon_3 = 1 \}.
$$

The action of $G_1$ on $E_1 \times E_2 \times E_3$ is given by:
\begin{equation*}
x_0 \mapsto \epsilon_0 x_0, \ u_0 \mapsto \eta_0 u_0, \ z_0 \mapsto \zeta_0 z_0,  \end{equation*}
\begin{equation*}
x_3 \mapsto \epsilon_1 x_3, \ u_3 \mapsto \epsilon_2 u_3, \ z_3 \mapsto \epsilon_3 z_3,  \  \epsilon_1 \epsilon_2 \epsilon_3 = 1.
\end{equation*}

2) \underline{$\lambda=1$:} then 
$\hat{X}:= \hat{X}_1$ is invariant under the group $G_0\cong (\ZZ / 2 \ZZ)^6 \leq (\ZZ / 2 \ZZ)^3 \times (\ZZ / 2 \ZZ)^3 \times (\ZZ / 2 \ZZ)^3$, where

$$
G_0:= \{ (\epsilon_0, \eta_1, \epsilon_1, \eta_0, \epsilon_2, \zeta_0, \epsilon_3) \subset (\ZZ / 2 \ZZ)^7 | \epsilon_1 \epsilon_2 \epsilon_3 = 1 \}.
$$

The action of $G_0$ on $E_1 \times E_2 \times E_3$ is given by:
\begin{equation*}
x_0 \mapsto \epsilon_0 x_0, \ u_0 \mapsto \eta_0 u_0, \ z_0 \mapsto \zeta_0 z_0, \begin{pmatrix}
x_1\\ u_1 \\ z_1
\end{pmatrix} \mapsto  \eta_1 \begin{pmatrix}
x_1\\ u_1 \\ z_1
\end{pmatrix} ,
\end{equation*}
\begin{equation*}
x_3 \mapsto \epsilon_1 x_3, \ u_3 \mapsto \epsilon_2 u_3, \ z_3 \mapsto \epsilon_3 z_3,  \  \epsilon_1 \epsilon_2 \epsilon_3 = 1.
\end{equation*}
\end{lemma}

\begin{proof}
Just observe that multiplication of the variables $x_1, u_1, z_1$ by $-1$ correspond to exchanging, for each $i=1,2,3$, $s_i $ with $t_i$.

\end{proof}

\begin{definition}
1) Let $G \cong (\ZZ / 2 \ZZ)^3 \leq G_1$ have the property  that $G$ acts freely on $\hat{X}_{\lambda}$. Then $S_{\lambda}:=\hat{X}_{\lambda} /G$ is  is called a {\em primary Burniat type surface}. 

2)  Let $G \cong (\ZZ / 2 \ZZ)^4 \leq G_0$, such that  $\hat{X}_1 = \hat{X}$ is invariant under $G$. We say that $G$ acts {\em 1-almost freely} on $\hat{X}$, if there is exactly one nontrivial element $g_0 \in G$ having isolated fixed points on $\hat{X}$ while  all  other  act freely.

 Then the minimal resolution of singularities $S$ of $X:=\hat{X} /G$   is called a {\em 4-nodal Burniat type surface}. 
\end{definition}
Observe that a primary Burniat type surface $S_{\lambda}$ is a smooth minimal surface of general type with $p_g=0$ and $K_S^2 =6$. In particular, primary Burniat surfaces are primary Burniat type surfaces.

The minimal resolution of a  4-nodal Burniat type surface is then a minimal surface of general type with $K_S^2=3$, $p_g=0$.

\section{The fixed points of $G_0$ on $\hat{X}$}

\begin{rem}
Fix $a_1,a_2,a_3 \in \CC$ distinct  so that the curve
$$
E:= \{ (x_0:x_1:x_2:x_3) \in \PP^3 | x_1^2 + x_2^2 + x_3^2 = 0, x_0^2 = a_1x_1^2 +a_2x_2^2+a_3x_3^2 \}
$$
is smooth (hence an elliptic curve). Then: 
$$
g(x_0:x_1:x_2:x_3) := (\alpha_0x_0, \alpha_1x_1, x_2, \alpha_3x_3), \ \alpha_i \in \{\pm 1\},
$$
has fixed points on $E$ if and only if either
\begin{itemize}
\item $\alpha_0 = \alpha_1 = \alpha_3 = -1$, or
\item exactly one $\alpha_i = -1$, the others are equal to 1.
\end{itemize}
\end{rem}

Note that the group of automorphisms that we consider is isomorphic to
 $(\ZZ / 2 \ZZ)^3 \cong \{(m_1,m_2,m_3,m_4) \in (\ZZ / 2 \ZZ)^4 | m_3=0 \}$.

\subsection{Elements of $G_0$ having a fixed locus of dimension 2 on $E_1 \times E_2 \times E_3$}

Let $g \in G_0$ be an element leaving a surface $$S \subset T:= E_1 \times E_2 \times E_3$$ pointwise fixed. Then we have the following three possibilities:
\begin{itemize}
\item[i)] $g= \id_{E_1} \times \id_{E_2} \times g_3$, where $g_3$ has fixed points on $E_3$;
\item[ii)] $g= \id_{E_1} \times g_2 \times  \id_{E_3}$, where $g_2$ has fixed points on $E_2$;
\item[iii)] $g= g_1 \times \id_{E_2} \times \id_{E_3}$, where $g_1$ has fixed points on $E_1$.
\end{itemize}

\noindent
i) $g= \id_{E_1} \times \id_{E_2} \times g_3$: this implies $\epsilon_0 = \eta_1 = \epsilon_1 = 1$ and $\eta_0 = \epsilon_2 = 1$. This implies $\epsilon_3 = 1$, whence we have for $g_3$ only one possibility:

$$
g_3 = \begin{pmatrix} -1 \\  1\\ 1 \\ 1   \end{pmatrix} .
$$
\noindent
By symmetry we get for the cases ii) and iii) the following two respective possibilities:

$$
ii) \ \ g_2 = \begin{pmatrix} -1 \\  1\\ 1 \\ 1   \end{pmatrix}, \ \ 
iii) \ \ g_1 = \begin{pmatrix} -1 \\  1\\ 1 \\ 1   \end{pmatrix}.
$$

\subsection{Elements of $G_0$ having a fixed locus of dimension 1 on $E_1 \times E_2 \times E_3$}

Let $g \in G_0$ be an element leaving a curve $C \subset T:= E_1 \times E_2 \times E_3$ pointwise fixed. Then we have the following three possibilities:
\begin{itemize}
\item[i)] $g= \id_{E_1} \times g_2 \times g_3$, where $g_2, g_3$ have fixed points on $E_2$ resp. $E_3$;
\item[ii)] $g= g_1 \times \id_{E_2}\times g_3$, where $g_1, g_3$ have fixed points on $E_1$ resp. $E_3$;
\item[iii)] $g= g_1 \times g_2 \times \id_{E_3}$, where $g_1, g_2$ have fixed points on $E_1$ resp. $E_2$.
\end{itemize}

\noindent
i) $g= \id_{E_1} \times g_2 \times g_3$: then $\epsilon_0 = \eta_1 = \epsilon_1 = 1$, in particular, $\epsilon_2 = \epsilon_3$. We have therefore:
$$
g = (g_1,g_2,g_3) = (\begin{pmatrix} 1 \\ 1\\ 1 \\ 1 \end{pmatrix}, \begin{pmatrix} \eta_0 \\ 1\\ 1 \\ \epsilon_2 \end{pmatrix}, \begin{pmatrix} \zeta_0 \\ 1\\ 1 \\ \epsilon_3 \end{pmatrix}).$$
This shows that we have the following two possibilities for $g_2$:
\begin{itemize}
\item[a)] $\eta_0 = 1$ and $\epsilon_2 = \epsilon_3 =-1$, 
\item[b)] $\eta_0 = -1$ and $\epsilon_2 = \epsilon_3 =1$.
\end{itemize}
a) The first  possibility  for $g$ is:
$$
g = (\begin{pmatrix} 1 \\ 1\\ 1 \\ 1 \end{pmatrix}, \begin{pmatrix} 1 \\ 1\\ 1 \\ -1 \end{pmatrix}, \begin{pmatrix} 1 \\ 1\\ 1 \\ -1\end{pmatrix}).
$$
b)  The second  possibility  for $g$ is:
$$
g = (\begin{pmatrix} 1 \\ 1\\ 1 \\ 1 \end{pmatrix}, \begin{pmatrix} -1 \\ 1\\ 1 \\ 1 \end{pmatrix}, \begin{pmatrix} -1 \\ 1\\ 1 \\ 1\end{pmatrix}).$$

\noindent
ii) $g= g_1 \times \id_{E_2}\times g_3$: by symmetry of $E_1$ and $E_2$, we get the following two possibilities for $g$:
$$
g = (\begin{pmatrix} 1 \\ 1\\ 1 \\ -1 \end{pmatrix}, \begin{pmatrix} 1 \\ 1\\ 1 \\ 1 \end{pmatrix}, \begin{pmatrix} 1 \\ 1\\ 1 \\ -1\end{pmatrix}), \ \rm{or} \ g = (\begin{pmatrix} -1 \\ 1\\ 1 \\ 1 \end{pmatrix}, \begin{pmatrix} 1 \\ 1\\ 1 \\ 1 \end{pmatrix}, \begin{pmatrix} -1 \\ 1\\ 1 \\ 1 \end{pmatrix}).$$

\noindent
iii) $g= g_1 \times g_2 \times \id_{E_3}$: again by symmetry we have two possibilities for $g$:
$$
g = (\begin{pmatrix} 1 \\ 1\\ 1 \\ -1 \end{pmatrix}, \begin{pmatrix} 1 \\ 1\\ 1 \\ -1 \end{pmatrix}, \begin{pmatrix} 1 \\ 1\\ 1 \\ 1\end{pmatrix}), \ \rm{or} \ g = (\begin{pmatrix} -1 \\ 1\\ 1 \\ 1 \end{pmatrix}, \begin{pmatrix} -1 \\ 1\\ 1 \\ 1 \end{pmatrix}, \begin{pmatrix} 1 \\ 1\\ 1 \\ 1 \end{pmatrix}).$$

\begin{rem}
Note that $\hat{X} \subset  T$ is an ample divisor, hence the fixed locus of the above elements has non trivial intersection with $\hat{X}$.
\end{rem}

\subsection{Elements of $G_0$ having isolated fixed points on $E_1 \times E_2 \times E_3$}
We still have to find all elements of $G_0$ which have isolated fixed points on $T$. 

An element $g = (g_1,g_2,g_3) \in G_0$ has isolated fixed points on $T=E_1 \times E_2 \times E_3$ if and only if $g_i (\neq \id)$ has  fixed points on $E_i$. Therefore,  on each $E_i$, $g_i$ is one of  the four elements
$$
g_i \in \{  \begin{pmatrix} -1 \\  1\\ 1 \\ 1  \end{pmatrix},  \begin{pmatrix} 1 \\ - 1\\ 1 \\ 1  \end{pmatrix},  \begin{pmatrix}  1\\ 1 \\ 1 \\ -1  \end{pmatrix},  \begin{pmatrix} -1 \\  -1\\ 1 \\ -1   \end{pmatrix}\}.
$$

We will list in  table \ref{fpburniat}  all the elements of $G_0$ which have fixed points on $T$. 

Observe  that, unlike before, we write the group additively.

\begin{table}
\caption{The elements of $G_0$ having fixed points on $T$, written additively.}
\label{fpburniat}
\renewcommand{\arraystretch}{1,3}
 \tabcolsep 2.8pt
\begin{tabular}{|c||c|c|c||c|c|c|c|c|c||c|c|c|c|c|c|c|c||}
\hline
&1&2&3&4&5&6&7&8&9&10&11&12&13&14&15&16&17 \\
\hline\hline
$\epsilon_0$&0&0&1&0&0&0&1&0&1&1&1&0&0&0&0&1&1\\
$\eta_1$&0&0&0&0&0&0&0&0&0&0&0&1&1&0&0&1&1\\
$\epsilon_1$&0&0&0&0&0&1&0&1&0&0&0&0&0&1&1&1&1\\
$\eta_0$&0&1&0&0&1&0&0&0&1&1&0&0&1&1&0&0&1\\
$\epsilon_2$&0&0&0&1&0&0&0&1&0&0&1&0&1&0&1&0&1\\
$\zeta_0$&1&0&0&0&1&0&1&0&0&1&0&0&1&0&1&1&0\\
$\epsilon_3$&0&0&0&1&0&1&0&0&0&0&1&0&1&1&0&1&0\\
 \hline
 
   \end{tabular}
\end{table}

More precisely,  the elements 1,2,3 have a fixed locus of dimension 2, the elements 4 to 9 have a fixed curve and the elements 10 to 17 have isolated fixed points on $T$.

We shall prove now the following
\begin{prop}
The elements $11 - 17$ do have fixed points on $\hat{X}$, whereas the fixed points of the element 10 do not intersect $\hat{X}$.
\end{prop}

\begin{proof}
We recall that we have the Del Pezzo surface $Y \subset \PP^1 \times \PP^1 \times \PP^1$,  given in the coordinates $(s_i : t_i)$ by $Y:= \{s_1s_2s_3 = t_1t_2t_3\}$, or in the coordinates $(x_i,u_i,z_i)$ as the subvariety of $\PP^2 \times \PP^2 \times \PP^2$ defined by the equations in lemma \ref{8equations}.

We have to check whether the fixed points of the elements $10-17$ listed in table \ref{fpburniat} are contained in the pull back $\hat{X}$ of $Y$.

\noindent
10) The fixed points are given by  $x_0 = u_0 = z_0$, i.e. they are of the form
\begin{small}
$$
((0:\pm i \mu_1 x_2:x_2:\pm \lambda_1 x_2),(0:\pm i \mu_2 u_2:u_2:\pm \lambda_2 u_2)(0:\pm i \mu_3 z_2:z_2:\pm \lambda_3 z_2)),
$$
\end{small}
where $\mu_i =\sqrt{1+\lambda_i^2}$, and $\lambda_i$ depends on $a_i$ (resp.$b_i$, resp $c_i$). 

It is now easy to check that, for a general choice of the elliptic curves $E_1, E_2, E_3$, points of this form never fulfill the 8 equations of $Y$.

\noindent
11) The fixed points are given by $x_0=u_3=z_3=0$. By remark \ref{paramfp} $u_3=z_3=0$ corresponds to $s_2t_2 = 0=s_3t_3=0$. Whence e.g. all points of the form 
$$
((s_1:t_1),(0:t_2),(s_3:0)),
$$
$(s_1:t_1)$ arbitrary, lie on $Y$. This implies that the pull-back of $Y$ contains fixed points of $G$ corresponding to number 11 in table \ref{fpburniat}.

\noindent
12) The fixed points are given by  $x_1 = u_1 = z_1=0$.  By remark \ref{paramfp} this correspond to $s_i = \pm t_i$. This implies that the points $s_i= \epsilon_it_i$, $\epsilon_i \in \{ \pm 1\}$, $\epsilon_1 \epsilon_2 \epsilon_3 = 1$, are contained in the pull-back of $Y$.

\noindent
13) Here we have $x_1=u_2=z_2=0$, or in the coordinates $(s_i:t_i)$: 
$$
s_1 = \pm t_1, \  t_2 = \pm i s_2, \ t_3 = \pm i s_3.
$$
Again it is obvious that some of these fixed points are contained in the pull-back of $Y$.

\noindent
14), 15) $x_3=u_0=z_3=0$ resp. $x_3=u_3=z_0=0$: these cases are equal  to  case 11 by symmetry on the three elliptic curves. Hence also here the fixed points are contained in the pull-back of $Y$.

\noindent
16), 17)  $x_2=u_1=z_2=0$ resp. $x_2=u_2=z_1=0$: these cases are symmetric to case 13.
\end{proof}

In the remaining part of the section we briefly sketch the analogous results for $G_1$, i.e., we exhibit the elements $g \in G_1$, which have fixed points on $E_1 \times E_2 \times E_3$. Recall that $G_1 \cong (\ZZ / 2 \ZZ)^5 \leq (\ZZ / 2 \ZZ)^3 \times (\ZZ / 2 \ZZ)^3 \times (\ZZ / 2 \ZZ)^3$, where

$$
G_1:= \{ (\epsilon_0,  \epsilon_1, \eta_0, \epsilon_2, \zeta_0, \epsilon_3) \subset (\ZZ / 2 \ZZ)^6 | \epsilon_1 \epsilon_2 \epsilon_3 = 1 \}.
$$

\begin{rem}
The calculations are quite the same as before for the group $G_0$, just note that here we always have $\eta_1 =1$. Then it is easy to see that the elements of $G_1$ having a fixed surface or a fixed curve are the same as for $G_0$. 
\end{rem}
For the elements having isolated fixed points there is a small difference.

\subsection{Elements of $G_1$ having isolated fixed points on $E_1 \times E_2 \times E_3$}
We have to find all elements of $G_1$, which have isolated fixed points on $Z$. We have to exclude those elements of $G_1$ from $G$, where some of the fixed points are contained in the base locus of the pencil $\hat{X}_{\lambda}$.

An element $g = (g_1,g_2,g_3) \in G_1$ has isolated fixed points on $Z=E_1 \times E_2 \times E_3$ if and only if $g_i  (\neq \id)$ has  fixed points on $E_i$.  On each $E_i$ we have  the two elements
$$
g_i =  \begin{pmatrix} -1 \\  1\\ 1 \\ 1  \end{pmatrix},   \begin{pmatrix}  1\\ 1 \\ 1 \\ -1  \end{pmatrix}.
$$

We will list now all elements of $G_1$ having fixed points on $Z$ in the following table \ref{primburniatfp}. 

Note  that again we write the group additively in the sequel.

\begin{table}
\caption{The elements of $G_1$ having fixed points on $Z$}
\label{primburniatfp}
\renewcommand{\arraystretch}{1,3}
 \tabcolsep 2.8pt
\begin{tabular}{|c||c|c|c||c|c|c|c|c|c||c|c|c|c||}
\hline
&1&2&3&4&5&6&7&8&9&10&11&12&13 \\
\hline\hline
$\epsilon_0$&0&0&1&0&0&0&1&0&1&1&1&0&0\\
$\eta_1$&0&0&0&0&0&0&0&0&0&0&0&0&0\\
$\epsilon_1$&0&0&0&0&0&1&0&1&0&0&0&1&1\\
$\eta_0$&0&1&0&0&1&0&0&0&1&1&0&1&0\\
$\epsilon_2$&0&0&0&1&0&0&0&1&0&0&1&0&1\\
$\zeta_0$&1&0&0&0&1&0&1&0&0&1&0&0&1\\
$\epsilon_3$&0&0&0&1&0&1&0&0&0&0&1&1&0\\
 \hline
 
   \end{tabular}
\end{table}

Note that the elements 1,2,3 have a fixed locus of dimension 2, the elements 4 to 9 have a fixed curve and the elements 10 to 13 have isolated fixed points on $Z$.

The following is easy to verify
\begin{prop}
The elements $11 - 13$ do have fixed points on the base locus of the pencil $\hat{X}_{\lambda}$, whereas the fixed points of the element 10 do not lie on the base locus of $\hat{X}_{\lambda}$.
\end{prop}

We can now prove the following

\begin{theo}
Let $S$ be a primary Burniat type surface. Then $S$ is a primary Burniat surface.
\end{theo}

\begin{proof}
The following MAGMA script shows that there are two subgroups $G \leq G_1$ acting freely on $\hat{X}_{\lambda}$, for $\lambda \in \CC$ general.

\begin{verbatim}
K:=FiniteField(2); V5:=VectorSpace(K,5); V2:=VectorSpace(K,2);
H:=Hom(V5,V2);
U1:=sub<V5|[0,0,0,0,1]>; U2:=sub<V5|[0,0,1,0,0]>;
U3:=sub<V5|[1,0,0,0,0]>; U4:=sub<V5|[0,0,0,1,0]>;
U5:=sub<V5|[0,0,1,0,1]>; U6:=sub<V5|[0,1,0,0,0]>;
U7:=sub<V5| [0,1,0,1,0] >; U8:=sub<V5| [1,0,1,0,0] >;
U9:=sub<V5| [1,0,0,0,1] >; U10:=sub<V5| [1,0,0,1,0] >;
U11:=sub<V5| [0,1,1,0,0] >; U12:=sub<V5| [0,1,0,1,1] >;
N:=sub<V5|[0,0,0,0,0]>;
w1:=V5![1,0,0,0,0];
w2:=V5![0,0,1,0,0];
x:=V2![1,0]; y:=V2![0,1];
M:={@ @};
  for a in H do 
    if a(w1) eq x then 
      if a(w2) eq y then
        if Kernel(a) meet U1 eq N then
          if Kernel(a) meet U2 eq N then
            if Kernel(a) meet U3 eq N then 
               if Kernel(a) meet U4 eq N then
                 if Kernel(a) meet U5 eq N then
                   if Kernel(a) meet U6 eq N then
                     if Kernel(a) meet U7 eq N then
                       if Kernel(a) meet U8 eq N then
                         if Kernel(a) meet U9 eq N then
                           if Kernel(a) meet U10 eq N then 
                             if Kernel(a) meet U11 eq N then 
                               if Kernel(a) meet U12 eq N then
                            Include(~M,a);
end if;end if;end if;end if;end if;end if;end if;
end if;end if;end if;end if;end if;end if;end if;
end for;
M;
{@
    [1 0]
    [1 1]
    [0 1]
    [0 1]
    [1 1],

    [1 0]
    [1 0]
    [0 1]
    [1 1]
    [1 1]
@}
\end{verbatim}
It is now easy to see that the two cases are equivalent under the symmetry exchanging $E_1$ and $E_2$. Therefore they yield the same surfaces.

\end{proof}

\section{4-Nodal Burniat type surfaces}
In this section we shall give a complete classification of 4-nodal Burniat type surfaces.

Recall  diagram (\ref{burniatdiagr}):

\begin{equation*}
\xymatrix{
E_1 \times E_2 \times E_3 \ar[dd]_{\mathcal{H}':=(\ZZ/2)^3}^{\pi'}& E_1: x_1^2 + x_2^2 + x_3^2 = 0, \  x_0^2 = a_1x_1^2 + a_2x_2^2+a_3x_3^2\\
&E_2: u_1^2 + u_2^2 + u_3^2 = 0, \ u_0^2 = b_1u_1^2 + b_2u_2^2+b_3u_3^2\\
P_1:=\PP^1 \times \PP^1 \times \PP^1\ar[dd]^{\pi }_{\mathcal{H}:=((\ZZ/2)^2)^3} & E_3: z_1^2 +z_2^2 +z_3^2 = 0,  \ z_0^2 = c_1 z_1^2 + c_2z_2^2 + c_3 z_3^2\\
&\\
P_2:=\PP^1 \times \PP^1 \times \PP^1
}
\end{equation*}

Using the notation in section \ref{actions} we see that 
 $\hat{X}:= \pi'^{-1}(Y)$ is invariant under $G_0 \cong (\ZZ /2 \ZZ)^6 \subset (\ZZ /2 \ZZ)^{9}$. It is now our aim to find 
subgroups $G \cong (\ZZ/2\ZZ)^4 \subset G_0$ such that there is exactly one element $g \in G$ having (isolated) fixed points on $\hat{X}$ and all the other nontrivial elements of $G$ act freely.

\begin{rem}
We shall see then that this unique element $g \in G$ has 32 fixed points on $\hat{X}$, whence $X:=\hat{X}/G$ has 4 nodes
(this fact justifies our terminology). 

It is clear that the minimal model $S$ of $X$ is a surface of general type with invariants $K_S^2 = 3$ and $\chi(S) = 1$. Looking in fact at the respective groups $G$ , we see that in all cases $q(S) = 0$, whence $p_g(S) = 0$.
\end{rem}

The following MAGMA script has as output bases of subgroups $G \leq G_0$ as $\FF_2$-vectorspaces, which contain exactly one element $g_0$ having fixed points on $\hat{X}$.

\begin{verbatim}
K:=FiniteField(2);
V6:=VectorSpace(K,6); V2:=VectorSpace(K,2); H:=Hom(V6,V2);
U1:=sub<V6|[0,0,0,0,0,1]>; U2:=sub<V6|[0,0,0,1,0,0]>;
U3:=sub<V6|[1,0,0,0,0,0]>; U4:=sub<V6|[0,0,0,0,1,0]>;
U5:=sub<V6|[0,0,0,1,0,1]>; U6:=sub<V6|[0,0,1,0,0,0]>;
U7:=sub<V6|[1,0,0,0,0,1]>; U8:=sub<V6|[0,0,1,0,1,0]>;
U9:=sub<V6|[1,0,0,1,0,0]>; U10:=sub<V6|[1,0,0,0,1,0]>;
U11:=sub<V6|[0,0,1,1,0,0]>; U12:=sub<V6|[0,0,1,0,1,1]>;
U13:=sub<V6|[0,1,0,0,0,0]>;  U14:=sub<V6|[0,1,0,1,1,1]>;  
U15:=sub<V6|[1,1,1,0,0,1]>; U16:=sub<V6|[1,1,1,1,1,0]>;
N:=sub<V6|[0,0,0,0,0,0]>; 
w1:=V6![1,0,0,0,0,0]; w2:=V6![0,0,0,1,0,0];
x:=V2![1,0]; y:=V2![0,1];
M:={@ @};
for a in H do 
  if a(w1) eq x then 
    if a(w2) eq y then
      if Kernel(a) meet U1 eq N then
        if Kernel(a) meet U2 eq N then
          if Kernel(a) meet U3 eq N then 
            if Kernel(a) meet U4 eq N then
              if Kernel(a) meet U5 eq N then
                if Kernel(a) meet U6 eq N then
                  if Kernel(a) meet U7 eq N then
                    if Kernel(a) meet U8 eq N then
                      if Kernel(a) meet U9 eq N then Include(~M,a);
end if; end if;end if;end if; end if; end if;
end if;end if;end if;end if;end if;
end for;
F:={@ V6! [1,0,0,0,1,0],V6! [0,0,1,1,0,0], 
V6! [0,0,1,0,1,1],V6! [0,1,0,0,0,0], 
V6![0,1,0,1,1,1],V6![1,1,1,0,0,1],V6![1,1,1,1,1,0] @};

M1:={@ @};
  for i in [1..24] do K:={@ @};
    for x in Kernel(M[i]) do  Include(~K,x); 
  end for;
    if  #(K meet F) eq 1 then Include(~M1,i);
  end if; end for;
M1;
{@ 7, 8, 10, 11, 13, 14, 15, 16, 17, 18, 19, 20, 22, 24 @}
MM:={@ M[7],M[8],M[10],M[11],M[13],M[14],M[15],M[16],M[17],
M[18],M[19],M[20],M[22],M[24] @};

MB:={@ MM[1], MM[2],MM[3],MM[5],MM[6],MM[7],MM[8] @};
L:={@ @};
  for x in MB do Include(~L,Kernel(x)); 
end for;
\end{verbatim}

\begin{rem}
We want to observe that MM contains 14 subgroups, which split into 7 pairs of equivalent subgroups under the symmetry obtained by exchanging $E_1$ and $E_2$.
\end{rem}
There are 7 groups in the set $L$.  We list generators for each of these in  table \ref{burniatk3}.

\begin{table}
\caption{Generators of $G \cong (\ZZ / 2 \ZZ)^4$}
\label{burniatk3}
\renewcommand{\arraystretch}{1,3}
\small \begin{tabular}{|c|ccc|ccc|ccc||c|ccc|ccc|ccc|}
\hline
&$\epsilon_0$&$\eta_1$&$\epsilon_1$&$\eta_0$&$\eta_1$&$\epsilon_2$&$\zeta_0$&$\eta_1$&$\epsilon_3$&&$\epsilon_0$&$\eta_1$&$\epsilon_1$&$\eta_0$&$\eta_1$&$\epsilon_2$&$\zeta_0$&$\eta_1$&$\epsilon_3$\\
\hline\hline
A& 1&0&0&0&0&1&0&0&1&B& 1&0&0&0&0&1&0&0&1\\
& 0&1&0&0&1&1&1&1&1&& 0&1&0&0&1&0&1&1&0\\
& 0&0&1&0&0&0&1&0&1&& 0&0&1&0&0&0&1&0&1\\
& 0&0&0&1&0&1&1&0&1&& 0&0&0&1&0&1&1&0&1\\
 \hline
 C& 1&0&0&0&0&1&1&0&1&D& 1&0&0&0&0&1&1&0&1\\
& 0&1&0&0&1&1&1&1&1&& 0&1&0&0&1&0&0&1&0\\
& 0&0&1&0&0&1&1&0&0&& 0&0&1&0&0&0&1&0&1\\
& 0&0&0&1&0&1&0&0&1&& 0&0&0&1&0&1&0&0&1\\
 \hline
 E& 1&0&0&0&0&1&1&0&1& F& 1&0&0&0&0&1&1&0&1\\
& 0&1&0&0&1&1&1&1&1&& 0&1&0&0&1&1&0&1&1\\
& 0&0&1&0&0&0&1&0&1&& 0&0&1&0&0&0&1&0&1\\
& 0&0&0&1&0&1&0&0&1&& 0&0&0&1&0&1&0&0&1\\
 \hline
 G& 1&0&0&0&0&1&1&0&1&&&&&&&&&&\\
& 0&1&0&0&1&0&1&1&0 &&&&&&&&&&\\
& 0&0&1&0&0&0&1&0&1&&&&&&&&&&\\
& 0&0&0&1&0&1&0&0&1&&&&&&&&&&\\
 \hline
 \hline
 \end{tabular}
\end{table}

\begin{rem}
In each of the 7 subgroups $G \leq G_0$ there is exactly one (non trivial) element having fixed points on $\hat{X}$. These elements are:
\begin{itemize}
\item[A)] $g_0 = (1,0,0,0,0,1,0,0,1)$,
\item[B)] $g_0 = (1,0,0,0,0,1,0,0,1)$,
\item[C)] $g_0 = (0,0,1,0,0,1,1,0,0)$,    
\item[D)] $g_0 = (0,1,0,0,1,0,0,1,0)$,
\item[E)] $g_0 = (1,1,1,0,1,0,1,1,1)$,
\item[F)] $g_0 = (1,1,1,1,1,1,0,1,0)$,
\item[G)] $g_0 = (0,1,0,1,1,1,1,1,1)$.
\end{itemize}
\end{rem}

In order to calculate the fundamental groups of the corresponding quotient  surfaces it is convenient  to rewrite the action of $G_i$, $i = A,B,C,D,E,F,G$, on $T:=E_1 \times E_2 \times E_3$ in terms of  uniformizing parameters $z_i$ for $E_i$. 

For 
$i \in \{1,2,3\}$, let $E_i := \CC / \langle 1, \tau_i \rangle$ 
be 
a complex elliptic curve.  Then we choose as basis for the $(\ZZ / 2 \ZZ)^3$-action on  $E_i$:
\begin{itemize}
\item $(z_i \mapsto -z_i) = (1,0,0)$,
\item $(z_i \mapsto -z_i + \frac{\tau_i}{2}) = (0,1,0)$,
\item $(z_i \mapsto -z_i + \frac{1}{2}) = (0,0,1)$.
\end{itemize}

Then we can rewrite the generators of $G_i$, $i \in \{ A,B,C,D,E,F,G \}$,  in table (\ref{burniatk3}) in the following way.

We would like to point out that in the cases $C,D,E,F,G$ we choose a different basis from the one in table (\ref{burniatk3}).

\begin{enumerate}
\item $G_A$ is generated by:
\begin{itemize}
\item[] $ g_1(z_1,z_2,z_3) = 
(-z_1, -z_2 + \frac 12,- z_3+\frac 12)$,
\item[] $g_2(z_1,z_2,z_3) 
= (-z_1+ \frac{\tau_1}{2}, z_2 + \frac 12+ \frac{\tau_2}{2}, -z_3 +\frac 12+ \frac{\tau_3}{2})$,
\item[] $g_3(z_1,z_2,z_3) = (-z_1+\frac 12, z_2, z_3 + \frac 12)$,
\item[]  $g_4(z_1,z_2,z_3) = (z_1, z_2 + \frac 12, z_3+ \frac 12)$.
\end{itemize}

\item $G_B$ is generated by:
\begin{itemize}
\item[] $ g_1(z_1,z_2,z_3) = 
(-z_1, -z_2 + \frac 12,- z_3+\frac 12)$,
\item[] $g_2(z_1,z_2,z_3) 
= (-z_1+ \frac{\tau_1}{2}, -z_2 + \frac{\tau_2}{2}, z_3 +\frac{\tau_3}{2})$,
\item[] $g_3(z_1,z_2,z_3) = (-z_1+\frac 12, z_2, z_3 + \frac 12)$,
\item[]  $g_4(z_1,z_2,z_3) = (z_1, z_2 + \frac 12, z_3+ \frac 12)$.
\end{itemize}

\item $G_C$ is generated by:
\begin{itemize}
\item[] $ g_1(z_1,z_2,z_3) = 
(-z_1 + \frac 12, -z_2 + \frac 12,- z_3)$,
\item[] $g_2(z_1,z_2,z_3) 
= (-z_1+ \frac{\tau_1}{2}, z_2 + \frac 12+ \frac{\tau_2}{2}, -z_3 +\frac 12+ \frac{\tau_3}{2})$,
\item[] $g_3(z_1,z_2,z_3) = (z_1+\frac 12, z_2, -z_3 + \frac 12)$,
\item[]  $g_4(z_1,z_2,z_3) = (z_1, z_2 + \frac 12, -z_3+ \frac 12)$.
\end{itemize}

\item $G_D$ is generated by:
\begin{itemize}
\item[] $ g_1(z_1,z_2,z_3) = 
(z_1 + \frac 12, -z_2 + \frac 12,z_3)$,
\item[] $g_2(z_1,z_2,z_3) 
= (-z_1+ \frac 12, z_2, z_3 +\frac 12)$,
\item[] $g_3(z_1,z_2,z_3) = (z_1, z_2 + \frac 12, -z_3 + \frac 12)$,
\item[]  $g_4(z_1,z_2,z_3) = (-z_1 + \frac{\tau_1}{2}, -z_2 +\frac{\tau_2}{2}, -z_3+ \frac{\tau_3}{2})$.
\end{itemize}

\item $G_E$ is generated by:
\begin{itemize}
\item[] $ g_1(z_1,z_2,z_3) = 
(z_1 + \frac 12, -z_2 + \frac 12,z_3)$,
\item[] $g_2(z_1,z_2,z_3) 
= (-z_1+ \frac 12, z_2, z_3 +\frac 12)$,
\item[] $g_3(z_1,z_2,z_3) = (z_1, z_2+\frac 12, -z_3 + \frac 12)$,
\item[]  $g_4(z_1,z_2,z_3) =(-z_1+ \frac 12 +\frac{\tau_1}{2},- z_2 + \frac{\tau_2}{2}, -z_3 +\frac 12+ \frac{\tau_3}{2})$.
\end{itemize}

\item $G_F$ is generated by:
\begin{itemize}
\item[] $ g_1(z_1,z_2,z_3) = 
(z_1 + \frac 12, -z_2 + \frac 12,z_3)$,
\item[] $g_2(z_1,z_2,z_3) 
= (-z_1+ \frac 12, z_2, z_3 +\frac 12)$,
\item[] $g_3(z_1,z_2,z_3) = (z_1, z_2+\frac 12, -z_3 + \frac 12)$,
\item[]  $g_4(z_1,z_2,z_3) =(-z_1+ \frac 12 +\frac{\tau_1}{2},- z_2 + \frac 12 + \frac{\tau_2}{2}, -z_3 +\frac{\tau_3}{2})$.
\end{itemize}

\item $G_G$ is generated by:
\begin{itemize}
\item[] $ g_1(z_1,z_2,z_3) = 
(z_1 + \frac 12, -z_2 + \frac 12,z_3)$,
\item[] $g_2(z_1,z_2,z_3) 
= (-z_1+ \frac 12, z_2, z_3 +\frac 12)$,
\item[] $g_3(z_1,z_2,z_3) = (z_1, z_2+\frac 12, -z_3 + \frac 12)$,
\item[]  $g_4(z_1,z_2,z_3) =(-z_1 +\frac{\tau_1}{2},- z_2 + \frac 12 +\frac{\tau_2}{2},- z_3 +\frac 12+ \frac{\tau_3}{2})$.
\end{itemize}
\end{enumerate}

\begin{rem}
1) We have $G_i \leq G_0 \leq (\ZZ/2 \ZZ)^3 \times (\ZZ/2 \ZZ)^3 \times (\ZZ/2 \ZZ)^3$. Denote by $K_i$, $i = 1,2,3$, the kernel of the projection on the i-th factor. Then we have:
\begin{itemize}
\item[i)] $K_3 \subset K_1 \oplus K_2$, for the groups $G_i$, $i =A,B,C$;
\item[ii)] $K_i \cap (K_j \oplus K_l) = \{ 0\}$, for $\{i,j,l\} = \{1,2,3 \}$, for the groups $G_i$, $i= D,E,F,G$.
\end{itemize}

\noindent 
2) The unique element in $G_i$ having fixed points on $\hat{X}$ is
\begin{itemize}
\item[i)] $g_1(z_1,z_2,z_3) =(-z_1, -z_2 + \frac 12,- z_3+\frac 12)$, for $i= A,B$,
\item[ii)] $g_1(z_1,z_2,z_3) =(-z_1 + \frac 12, -z_2 + \frac 12,- z_3)$, for $i= C$,
\item[iii)] $g_4$, for $i=D,E,F,G$.
\end{itemize}
\end{rem}

Recall that we have written $E_i=\CC / \langle e_i, \tau_i e_i\rangle$, $i=1,2,3$. 

Denote by $\Lambda$ the fundamental group of $E_1 \times E_2 \times E_3$, so that, setting
 $\Lambda_i = \langle e_i, \tau_i e_i \rangle$, we have  $\Lambda = \Lambda_1 \oplus \Lambda_2 \oplus \Lambda_3$.
 
 At this moment we invoke the hyperplane section theorem of Lefschetz, which we apply to
 the ample divisor  $\hat{X} \subset E_1 \times E_2 \times E_3$: it follows that
 $\pi_1 (\hat{X}) \cong \pi_1 ( E_1 \times E_2 \times E_3) = \Lambda$.

 Hence  the universal covering  $\tilde{X} $ of $\hat{X} \subset E_1 \times E_2 \times E_3$
has natural inclusion $\tilde{X} \subset \CC^3$. 

Now  the affine group 
\begin{equation}
\Gamma_i := \langle \gamma_1,\gamma_2, \gamma_3, \gamma_4, e_1, \tau_1e_1, e_2,\tau_2e_2, e_3, \tau_3 e_3 \rangle \leq \mathbb{A}(3,\CC),
\end{equation}
where the $\gamma_k$ are lifts of the generators $g_k$ of $G_i$ as affine transformations, acts on $\CC^3$ leaving $\tilde{X}$ invariant. 
 
Moreover, $X_i= \hat{X}/G_i = \tilde{X}/\Gamma_i$. 

Then by Armstrong's result 
(cf. \cite{armstrong1}, \cite{armstrong2}) we have
\begin{equation}
\pi_1 (X_i) = \Gamma_i / \Tors(\Gamma_i),
\end{equation}
where $\Tors(\Gamma_i)$ is the normal subgroup of $\Gamma_i$ generated by all elements of $\Gamma_i$ having
finite order (indeed they have order equal to $2$): since these are precisely the elements which have fixed points on $\tilde{X}$.

\begin{rem} 
Denote by $g_0 \in G$ the unique element which has fixed points on $\hat{X}$, and denote by $\gamma_0 \in \Gamma_i$ a lift of $g_0$ to ${\mathbb A}(3, \CC)$. Observe that 
$$
\gamma_0 \begin{pmatrix} z_1 \\ z_2\\ z_3  \end{pmatrix} = \begin{pmatrix} -z_1 + \mu_1\\ - z_2 + \mu_2\\  -z_3  +\mu_3\end{pmatrix}, 
$$
where $\mu_i = \frac 12 \epsilon_i \in \frac 12 \Lambda_i$.

\noindent
1) Assume that $\gamma \in \Gamma_i$ has a fixed point on the universal covering $\tilde{X}$ of $\hat{X}$. Then there is a $\lambda \in \Lambda$ such that $\gamma = \gamma_0 t_{\lambda}$.

\noindent
2) Let  $z = (z_1,z_2,z_3) \in \tilde{X} \subset \CC^3$. Then $z$ yields a fixed point of $g_0$ on $\hat{X}$ if and only if  there exists $\hat{\lambda} \in \Lambda$ such that
$$
2  \begin{pmatrix} z_1 \\ z_2\\ z_3  \end{pmatrix} =  \frac 12  \begin{pmatrix} \epsilon_1 \\ \epsilon_2\\ \epsilon_3  \end{pmatrix} + \hat{\lambda} \ \ \iff z = \frac 14 \epsilon + \frac 12 \hat{\lambda},
$$
where $\epsilon = \begin{pmatrix} \epsilon_1 \\ \epsilon_2\\ \epsilon_3  \end{pmatrix}$.
\end{rem}

We need the following
\begin{lemma}
$z= \frac 14 \epsilon + \frac 12 \hat{\lambda} \in \tilde{X}$ is a fixed point of $\gamma = \gamma_0 t_{\lambda}$ if and only if $\lambda = - \hat{\lambda}$.
\end{lemma}

\begin{proof}
\begin{multline}
\gamma(z) = \gamma_0(z+ \lambda) = -z- \lambda + \frac 12 \epsilon = - \frac 14 \epsilon - \frac 12 \hat{\lambda} - \lambda + \frac 12 \epsilon = \\
= \frac 14 \epsilon + \frac 12 \hat{\lambda} - \hat{\lambda} - \lambda = z - \hat{\lambda} - \lambda = z \ \iff \ \lambda = - \hat{\lambda}.
\end{multline}
\end{proof}

Note that $g_0$ has 64 fixed points on $T=E_1 \times E_2 \times E_3$, but only 32 lie on $\hat{X}$. These $32$ points are divided in four $G_i$ - orbits. Let $P_1, \ldots, P_4 \in \tilde{X}$ be four representatives of the four orbits. Then we have $P_i = \frac 14 \epsilon + \frac 12 \hat{\lambda}_{P_i}$. Then:
$$
\Tors(\Gamma_i) = \langle \langle \gamma_0 t_{\hat{\lambda}_{P_1}},  \gamma_0 t_{\hat{\lambda}_{P_2}},  \gamma_0 t_{\hat{\lambda}_{P_3}},  \gamma_0 t_{\hat{\lambda}_{P_4}} \rangle \rangle.
$$

Moreover, since the point $z$ can be changed modulo $\Lambda$, the above argument shows that 
$ 2  \Lambda \subset \Tors(\Gamma_i) $, hence $ \pi_1(S) = \pi_1 (X) $ is a quotient of the 2-step nilpotent   group $\Pi_G$
such that 
$$ 1 \ra \Lambda/ 2 \Lambda \ra \Pi_G  \ra G \ra 1 . $$ 
We can now  prove the following theorem:

\begin{theo}
Let $S_i$, $i \in \{A,B,C,D,E,F,G \}$ be the minimal resolution of the surface $X_i :=\hat{X}/G_i$ (having four ordinary nodes). Then $S_i$ is a minimal surface of general type with $K_{S_i}^2 = 3$, $p_g(S_i) = 0$, with fundamental group
\begin{itemize}
\item[i)] $\pi_1(S_i) \cong(\ZZ / 2 \ZZ)^2 \times \ZZ/ 4 \ZZ$, for $i =A, B,C$;
\item[ii)] $\pi_1(S_i) \cong \HH \times \ZZ / 2 \ZZ$, for $i =D$;
\item[iii)] $\pi_1(S_i) \cong SmallGroup(16,13)$, for $i =E,F,G$.
\end{itemize}
\end{theo}

\begin{rem}
1) Cases D,E,F,G are obviously  quotients of a primary Burniat surface by an involution having four isolated fixed points. 

Case A,B,C have the same fundamental group as the Keum-Naie surfaces with $K^2 = 3$. 

Case E,F,G yield a (3-dimensional) family, which is new. Actually, the fundamental group $SmallGroup(16,13)$ which is the central product of the dihedral group of order 8 and the cyclic group of order 4 has not yet been realized by a surface with $K^2 =3$, $p_g=0$.

2) Denote by $\hat{S}_i$ the double cover of $X_i$ branched exactly in the four nodes. Then 
\begin{itemize}
\item $\hat{S}_i$ is a surface of general type with $K_S^2 = 6$, $p_g=q=1$ if $i= A,B,C$,
\item $\hat{S}_i$ is a primary Burniat surface for $i=D,E,F,G$.
\end{itemize}

3) It is easy to see that the groups $G_A,G_B,G_C$ yield the same family of surfaces. Indeed, exchanging $E_2$ with $E_3$ has the effect  of exchanging $G_A$ and $G_B$, whereas exchanging $E_1$ with $E_3$ has the effect of exchanging $G_B$ and $G_C$.

The same holds for the groups $G_E,G_F$ and $G_G$. Therefore, in order to prove the above theorem, it suffices to calculate the fundamental group in the cases $A,D,E$.
\end{rem}
\begin{proof}
A)  The fixed points of $g_0(z_1,z_2,z_3) = (-z_1,-z_2+ \frac 12, -z_3 + \frac 12)$ are the points $(z_1, z_2, z_3) \in E_1 \times E_2 \times E_3$ such that 
$$
z_1 \in \{ 0, \frac 12, \frac{\tau_1}{2}, \frac 12 + \frac{\tau_1}{2} \},
$$
$$
z_i \in \{ \frac 14, \frac 14 + \frac 12, \frac 14 + \frac{\tau_i}{2}, \frac 14 + \frac 12 + \frac{\tau_i}{2} \},  \ i = 2,3.
$$
These are 64 points, but only 32 of these are on $\hat{X}$, namely:
\begin{multline}
(z_1,z_2,z_3), \ z_1 \in \{ 0, \frac 12, \frac{\tau_1}{2}, \frac 12 + \frac{\tau_1}{2} \}, \  (z_2, z_3) \in \{(\frac 14, \frac 14 + \frac{\tau_3}{2}), (\frac 14, \frac 14 + \frac 12+ \frac{\tau_3}{2}), \\
(\frac 14 + \frac 12 , \frac 14 + \frac{\tau_3}{2}), (\frac 14 + \frac 12, \frac 14 + \frac 12 + \frac{\tau_3}{2}), (\frac 14 + \frac{\tau_2}{2}, \frac 14), (\frac 14 + \frac{\tau_2}{2}, \frac 14 + \frac 12), \\
(\frac 14 + \frac 12 + \frac{\tau_2}{2}, \frac 14), (\frac 14 + \frac 12 + \frac{\tau_2}{2}, \frac 14 + \frac 12)\}.
\end{multline}

In fact, recall that the affine equation of $\hat{X}$ (cf. \cite{inoue}) is
$$
\hat{X} = \{(z_1,z_2,z_3) \in T | \mathcal{L}_1(z_1) \mathcal{L}_2(z_2) \mathcal{L}_3 (z_3) = b_1b_2b_3 \},
$$
where $b_i = \mathcal{L}_i( \frac{\tau_i}{4})$. Observe that $b_i^2 = a_i$.
Let $(\mathcal{L}_i(z_i)_0: \mathcal{L}_i(z_i)_1)$ be homogeneous coordinates of the point $\mathcal{L}_i(z_i)$. 
The equation of $\hat{X}$ is then:
$$
\mathcal{L}_1(z_1)_0 \mathcal{L}_2(z_2)_0 \mathcal{L}_3 (z_3)_0 = b_1b_2b_3\mathcal{L}_1(z_1)_1 \mathcal{L}_2(z_2)_1 \mathcal{L}_3 (z_3)_1.
$$
It follows easily from the properties of the Legendre function that 
$$
(\mathcal{L}_i(z_i + \frac{\tau_i}{2})_0: \mathcal{L}_i(z_i + \frac{\tau_i}{2})_1) = (a_i\mathcal{L}_i(z_i)_1: \mathcal{L}_i(z_i)_0).
$$

In particular, we have 
$$
(\mathcal{L}_i(\frac 14)_0: \mathcal{L}_i(\frac 14)_1) = (0:1), \ \ (\mathcal{L}_i(\frac 14 + \frac{\tau_i}{2})_0: \mathcal{L}_i(\frac 14 + \frac{\tau_i}{2})_1) = (1:0).
$$
Now it follows easily that a fixed point  $(z_1,z_2,z_3)$ of $g_0$ on $T$ lies in fact on $\hat{X}$ if and only if it satisfies the equations
$$
\mathcal{L}_1(z_1)_0 \mathcal{L}_2(z_2)_0 \mathcal{L}_3 (z_3)_0 = \mathcal{L}_1(z_1)_1 \mathcal{L}_2(z_2)_1 \mathcal{L}_3 (z_3)_1=0.
$$
Therefore a fixed point  $(z_1,z_2,z_3) \in T$ of $g_0$ lies on $\hat{X}$ if and only if $z_1 \in \{ 0, \frac 12, \frac{\tau_1}{2}, \frac 12 + \frac{\tau_1}{2} \}$ and

\begin{multline*}
(z_2, z_3) \in \{(\frac 14, \frac 14 + \frac{\tau_3}{2}), (\frac 14, \frac 14 + \frac 12+ \frac{\tau_3}{2}), 
(\frac 14 + \frac 12 , \frac 14 + \frac{\tau_3}{2}), (\frac 14 + \frac 12, \frac 14 + \frac 12 + \frac{\tau_3}{2}),\\ (\frac 14 + \frac{\tau_2}{2}, \frac 14), (\frac 14 + \frac{\tau_2}{2}, \frac 14 + \frac 12), 
(\frac 14 + \frac 12 + \frac{\tau_2}{2}, \frac 14), (\frac 14 + \frac 12 + \frac{\tau_2}{2}, \frac 14 + \frac 12)\}.
\end{multline*}

These points fall into 4 $G_A$- orbits, and it is easy to verify that we can choose as representatives the four points:
$$
P_1 = (0, \frac 14, \frac 14 + \frac{\tau_3}{2}), \ P_2 = (\frac 12, \frac 14, \frac 14 + \frac{\tau_3}{2}), 
$$
$$
\ P_3 = (\frac{\tau_1}{2}, \frac 14, \frac 14 + \frac{\tau_3}{2}), \ P_4 = (\frac 12 + \frac{\tau_1}{2}, \frac 14, \frac 14 + \frac{\tau_3}{2}).
$$
Writing as above $P_i = \frac 14 \epsilon + \frac 12 \hat{\lambda}_{P_i}$, we see that $\epsilon = \begin{pmatrix} 0 \\ 1 \\ 1 \end{pmatrix}$, and 
$$
\hat{\lambda}_{P_1} = \begin{pmatrix} 0 \\ 0 \\ \tau_3 \end{pmatrix}, \  \hat{\lambda}_{P_2} = \begin{pmatrix} 1 \\ 0 \\ \tau_3 \end{pmatrix}, \ \hat{\lambda}_{P_3} = \begin{pmatrix} \tau_1 \\ 0 \\ \tau_3 \end{pmatrix}, \hat{\lambda}_{P_4} = \begin{pmatrix} 1+ \tau_1 \\ 0 \\ \tau_3 \end{pmatrix}.
$$
Therefore 
$$
\pi_1(X_j) = \Gamma_i / \langle \langle \gamma_0 t_{\hat{\lambda}_{P_i}} : i=1,2,3,4 \rangle \rangle, \  j = A, B.
$$

The following MAGMA script gives $\pi_1(X_j) \cong (\ZZ / 2 \ZZ)^2 \times \ZZ / 4 \ZZ$.

\begin{verbatim}
G1:=DirectProduct([CyclicGroup(2),CyclicGroup(2),CyclicGroup(2)]);
G2:=DirectProduct([CyclicGroup(2),CyclicGroup(2),CyclicGroup(2)]);
G3:=DirectProduct([CyclicGroup(2),CyclicGroup(2),CyclicGroup(2)]);

H:=DirectProduct([G1,G2,G3]);
PolyGroup:=func<seq|Group<a1,a2,a3,a4|
           a1^seq[1], a2^seq[2],a3^seq[3],a4^seq[4], a1*a2*a3*a4>>;
P1:=PolyGroup([2,2,2,2]);
P2:=PolyGroup([2,2,2,2]);
P3:=PolyGroup([2,2,2,2]);
P:=DirectProduct([P1,P2,P3]);
f:=Homomorphism(P,H, [P.1,P.2,P.3,P.4,P.5,P.6,P.7,P.8,P.9,
P.10,P.11,P.12],[H!(1,2),H!(3,4),H!(5,6),H!(1,2)(3,4)(5,6),
H!(7,8),H!(9,10),H!(11,12),H!(7,8)(9,10)(11,12),H!(13,14),
H!(15,16),H!(17,18),H!(13,14)(15,16)(17,18)]);
R:=Rewrite(P,Kernel(f));
R;
Finitely presented group R on 6 generators
Generators as words in group P
    R.1 = (P.2 * P.1)^2   /* =  e_1
    R.2 = (P.3 * P.1)^2   /* = \tau_1
    R.3 = (P.6 * P.5)^2   /*= e_2
    R.4 = (P.7 * P.5)^2   /*= \tau_2
    R.5 = (P.10 * P.9)^2  /* = e_3
    R.6 = (P.11 * P.9)^2  /*= \tau_3
Relations
    (R.1, R.2^-1) = Id(R)
    (R.3, R.4^-1) = Id(R)
    (R.5, R.6^-1) = Id(R)
    (R.4^-1, R.6^-1) = Id(R)
    (R.1^-1, R.5^-1) = Id(R)
    (R.5, R.2) = Id(R)
    (R.1^-1, R.3^-1) = Id(R)
    (R.2^-1, R.4^-1) = Id(R)
    (R.1^-1, R.6^-1) = Id(R)
    (R.3^-1, R.6^-1) = Id(R)
    (R.4^-1, R.5^-1) = Id(R)
    (R.2^-1, R.6^-1) = Id(R)
    (R.3^-1, R.5^-1) = Id(R)
    (R.4, R.1) = Id(R)
    (R.2^-1, R.3^-1) = Id(R)
    R.6^-1 * R.5 * R.2^-1 * R.1 * R.5^-1 * R.6 * 
    R.1^-1 * R.2 = Id(R)
    R.1^-1 * R.2 * R.3^-1 * R.4 * R.2^-1 * 
    R.1 * R.4^-1 * R.3 = Id(R)
    R.3^-1 * R.4 * R.5^-1 * R.6 * R.4^-1 * 
    R.3 * R.6^-1 * R.5 = Id(R)
CASE A:
***********
GG1:=sub<H|H!(1,2)(11,12)(17,18), 
H!(3,4)(9,10)(11,12)(13,14)(15,16)(17,18), 
H!(5,6)(13,14)(17,18),H!(7,8)(11,12)(13,14)(17,18)>;

/*The only element of GG1 having fixed points is
(1,2)(11,12)(17,18).*/

Pi1:=Rewrite(P,GG1@@f);
Q1:=quo<Pi1|P.1*P.7*P.11, P.1*P.7*P.11*(P.11*P.9)^2,
P.1*P.7*P.11*(P.2*P.1)^2*(P.11*P.9)^2, 
P.1*P.7*P.11*(P.3*P.1)^2*(P.11*P.9)^2, 
P.1*P.7*P.11*(P.2*P.1)^2*(P.3*P.1)^2*(P.11*P.9)^2 >;
IdentifyGroup(Q1);
<16, 10>
\end{verbatim}

\noindent 
D) Here we have $g_0 = (-z_1+ \frac{\tau_1}{2}, -z_2+\frac{\tau_2}{2}, -z_3+\frac{\tau_3}{2})$.
The 64 fixed points of $g_0$ on $T:=E_1 \times E_2 \times E_3$ are:
$$
z \in  \{ \frac 14 \begin{pmatrix} \pm \tau_1 \\ \pm \tau_2 \\ \pm \tau_3 \end{pmatrix} + \frac 12 (\ZZ/2 \ZZ)^3 \}.
$$
Here it suffices again to look at the affine equation of $\hat{X}$ and we see that all the above points satisfy 

$$
\mathcal{L}_1(z_1) \mathcal{L}_2(z_2) \mathcal{L}_3 (z_3) = \pm b_1b_2b_3.
$$
They lie on $\hat{X}$ (i.e., they fulfill the equation $\mathcal{L}_1(z_1) \mathcal{L}_2(z_2) \mathcal{L}_3 (z_3) =  b_1b_2b_3$) if and only if
$$
z \in \{\frac 14 \begin{pmatrix} \pm \tau_1 \\ \pm \tau_2 \\ \pm \tau_3 \end{pmatrix} + \frac 12 \{ 0, \begin{pmatrix} 1 \\ 1 \\ 0 \end{pmatrix},  \begin{pmatrix} 1 \\ 0 \\ 1 \end{pmatrix}, \begin{pmatrix} 0 \\ 1 \\ 1 \end{pmatrix} \}\}.
$$

It is easy to see that we can choose as 
representatives for the 4 $G_D$-orbits:
$$
P_1 = (\frac{\tau_1}{4}, \frac{\tau_2}{4}, \frac{\tau_3}{4}), \ P_2 = (\frac{\tau_1}{4} + \frac 12, \frac{\tau_2}{4} +\frac 12, \frac{\tau_3}{4}), 
$$
$$
P_3 = (\frac{\tau_1}{4} + \frac 12, \frac{\tau_2}{4}, \frac{\tau_3}{4}+ \frac 12), \ P_4 = (\frac{\tau_1}{4}, \frac{\tau_2}{4} +\frac 12, \frac{\tau_3}{4} + \frac 12).
$$

Hence we have:
$$
\hat{\lambda}_{P_1} = 0, \  \hat{\lambda}_{P_2} =  \begin{pmatrix} 1 \\ 1 \\ 0 \end{pmatrix}, \ \hat{\lambda}_{P_3} = \begin{pmatrix} 1 \\ 0 \\ 1 \end{pmatrix}, \hat{\lambda}_{P_4} = \begin{pmatrix} 0 \\ 1 \\ 1  \end{pmatrix}.
$$

And the MAGMA script
\begin{verbatim}
CASE D 
**********
GG4:=sub<H|H!(1,2)(11,12)(13,14)(17,18), H!(3,4)(9,10)(15,16), 
H!(5,6)(13,14)(17,18),H!(7,8)(11,12)(17,18)>;

/*The only element of GG4 having fixed points is
(3,4)(9,10)(15,16).*/

Pi4:=Rewrite(P,GG4@@f);

Q4:=quo<Pi4| P.2*P.6*P.10, P.2*P.6*P.10*(P.2 * P.1)^2*(P.6 * P.5)^2,
P.2*P.6*P.10*(P.2 * P.1)^2 *(P.10 * P.9)^2, 
P.2*P.6*P.10*(P.6 * P.5)^2 *(P.10 * P.9)^2>;
IdentifyGroup(Q4);
<16, 12>
\end{verbatim}
gives $\pi_1(X_D) \cong \HH \times \ZZ/ 2\ZZ$.

\medskip

\noindent
E) Here we have $g_0 = (-z_1+ \frac 12 +\frac{\tau_1}{2}, -z_2+\frac{\tau_2}{2}, -z_3+\frac 12 +\frac{\tau_3}{2})$.
The 64 fixed points of $g_0$ on $T$ are:
$$
z \in \{ \frac 14 \begin{pmatrix} \pm (1+ \tau_1) \\ \pm \tau_2 \\ \pm (1+\tau_3) \end{pmatrix}  + \frac 12 (\ZZ / 2 \ZZ)^3\}.
$$
Observe now that 
$$
\mathcal{L}_i(\frac 14 + \frac{\tau_i}{4})^2 = \mathcal{L}_i(\frac 14 + \frac{\tau_i}{4} + \frac 12)^2 = -a_i,
$$
whence $\{\mathcal{L}_i(\frac 14 + \frac{\tau_i}{4}), \mathcal{L}_i(\frac 14 + \frac{\tau_i}{4} + \frac 12)\} = \{\sqrt{-1} b_i, -\sqrt{-1} b_i\}$.

Then we see that  the points 
$$
z \in \{\frac 14 \begin{pmatrix} \pm (1+ \tau_1) \\ \pm \tau_2 \\ \pm (1+\tau_3) \end{pmatrix} + \frac 12 \{ \begin{pmatrix} 0 \\ 0 \\ 1\end{pmatrix}, \begin{pmatrix} 1 \\ 1 \\ 1 \end{pmatrix},  \begin{pmatrix} 1 \\ 0 \\ 0 \end{pmatrix}, \begin{pmatrix} 0 \\ 1 \\ 0 \end{pmatrix} \}\}
$$
lie on $\hat{X}$, whereas the other 32 points satisfy the equation $\mathcal{L}_1(z_1) \mathcal{L}_2(z_2) \mathcal{L}_3 (z_3) = - b_1b_2b_3$.

We again can choose as representatives of the four $G_E$-orbits the following points:

$$
P_1= \frac 14 \begin{pmatrix} (1+ \tau_1) \\  \tau_2 \\  (1+\tau_3) \end{pmatrix} + \frac 12  \begin{pmatrix} 0 \\ 0 \\ 1\end{pmatrix}, 
\ P_2= \frac 14 \begin{pmatrix} (1+ \tau_1) \\  \tau_2 \\  (1+\tau_3) \end{pmatrix} + \frac 12  \begin{pmatrix} 1 \\ 1 \\ 1\end{pmatrix}, 
$$
$$
P_3= \frac 14 \begin{pmatrix} (1+ \tau_1) \\  \tau_2 \\  (1+\tau_3) \end{pmatrix} + \frac 12  \begin{pmatrix} 1 \\ 0 \\ 0\end{pmatrix}, 
\ P_4= \frac 14 \begin{pmatrix} (1+ \tau_1) \\  \tau_2 \\  (1+\tau_3) \end{pmatrix} + \frac 12  \begin{pmatrix} 0 \\ 1 \\ 0\end{pmatrix}, 
$$
whence we have
$$
\hat{\lambda}_{P_1} = \begin{pmatrix} 0 \\ 0 \\ 1\end{pmatrix}, \  \hat{\lambda}_{P_2} =  \begin{pmatrix} 1 \\ 1 \\ 1 \end{pmatrix}, \ \hat{\lambda}_{P_3} = \begin{pmatrix} 1 \\ 0 \\ 0 \end{pmatrix}, \hat{\lambda}_{P_4} = \begin{pmatrix} 0 \\ 1 \\ 0 \end{pmatrix}.
$$

And the MAGMA script
\begin{verbatim}
CASE E 
***********
GG5:=sub<H|H!(1,2)(11,12)(13,14)(17,18), 
H!(3,4)(9,10)(11,12)(13,14)(15,16)(17,18), 
H!(5,6)(13,14)(17,18),H!(7,8)(11,12)(17,18)>;

/*The only element of GG5 having fixed points is
(1, 2)(3, 4)(5, 6)(9, 10)(13, 14)(15, 16)(17, 18).*/

Pi5:=Rewrite(P,GG5@@f);
Q5:=quo<Pi5| P.1*P.2*P.3*P.6*P.9*P.10*P.11*(P.10*P.9)^2, 
P.1*P.2*P.3*P.6*P.9*P.10*P.11*(P.2 * P.1)^2*(P.6 * P.5)^2*(P.10*P.9)^2, 
P.1*P.2*P.3*P.6*P.9*P.10*P.11*(P.2 * P.1)^2 , 
P.1*P.2*P.3*P.6*P.9*P.10*P.11*(P.6 * P.5)^2>;
IdentifyGroup(Q5);
<16, 13>
\end{verbatim}
gives $\pi_1(X_D) \cong SmallGroup(16,13)$.
\end{proof}



\bigskip
\noindent {\bf Authors' Address:}

\noindent I.Bauer, F. Catanese \\ Lehrstuhl Mathematik VIII\\
Mathematisches Institut der Universit\"at Bayreuth,\\
   Universit\"at Bayreuth, NW II\\
Universit\"atsstr. 30\\ 95447 Bayreuth

\begin{verbatim}
ingrid.bauer@uni-bayreuth.de,
 fabrizio.catanese@uni-bayreuth.de
\end{verbatim}
\end{document}